\documentclass[a4paper,oneside,11pt]{article}
\usepackage{import}
\usepackage{fullpage}
\usepackage{fancyhdr}
\AtBeginEnvironment{pmatrix}{\setlength{\arraycolsep}{8pt}}
\usepackage{setspace}

\usepackage{amsmath,etoolbox}
\usepackage{amssymb}
\usepackage[utf8x]{inputenc} 
\usepackage{accents}
\usepackage{bbm}
\usepackage{mathtools}
\usepackage{amsthm}
\usepackage[english]{babel}
\usepackage{letltxmacro}
\usepackage{xparse}
\usepackage[dvipsnames,svgnames,x11names,hyperref]{xcolor}
\usepackage{graphicx}
\usepackage{multirow}
\usepackage{multicol}
\usepackage{nicefrac}

\usepackage{tikz}
\usetikzlibrary{decorations.pathreplacing}
\usetikzlibrary{decorations.markings}
\usepackage{wrapfig}
\usepackage{caption}
\usepackage{subcaption}
\usepackage[percent]{overpic}
\usepackage{tikz-cd}

\usepackage{ebgaramond}
\usepackage{euscript}
\usepackage{mathrsfs}
\usepackage{bm}
\usepackage{commath}
\usepackage{accents}
\usepackage[bbgreekl]{mathbbol}
\usepackage[euler]{textgreek}
\usepackage[scr=boondoxo]{mathalfa}
\DeclareSymbolFontAlphabet{\mathbbm}{bbold}
\DeclareSymbolFontAlphabet{\mathbb}{AMSb}%
\usepackage{mleftright}
\usepackage{dirtytalk}
\usepackage{shuffle}
\usepackage{scalerel}
\usepackage{tensor}
\usepackage{soul}
\setul{0pt}{0.4pt}
\usepackage[T1]{fontenc}

\usepackage[colorlinks=true]{hyperref}
\definecolor{imperialBlue}{RGB}{0, 62, 116}
\definecolor{imperialBrick}{RGB}{165,25,0}
\definecolor{imperialProcess}{RGB}{0,133,202}
\definecolor{imperialGreen}{RGB}{2,137,59}
\definecolor{imperialRed}{RGB}{221,37,1}
\definecolor{imperialOrange}{RGB}{210,64,0}
\definecolor{imperialBlue2}{RGB}{0,110,175}
\definecolor{imperialTangerine}{RGB}{236,115,0}
\definecolor{imperialPurple}{RGB}{101,48,152}
\definecolor{imperialLime}{RGB}{196,214,0}
\definecolor{imperialKermit}{RGB}{102,164,10}
\hypersetup{
    colorlinks,
    citecolor=imperialGreen,
    filecolor=imperialBlue,
    linkcolor=imperialBlue,
    urlcolor=imperialProcess
}

\usepackage[nameinlink, capitalise]{cleveref}
\usepackage{mdframed}
\newtheorem{theorem}{Theorem}[section]
\newtheorem{corollary}[theorem]{Corollary}
\newtheorem{lemma}[theorem]{Lemma}
\newtheorem{proposition}[theorem]{Proposition}
\newtheorem{remark}[theorem]{Remark}
\theoremstyle{definition}
\newtheorem{definition}[theorem]{Definition}
\newtheorem{assumption}[theorem]{Assumption}

\newmdtheoremenv[
hidealllines=true,
leftline=true,
innertopmargin=0pt,
innerbottommargin=0pt,
linewidth=4pt,
linecolor=gray!40,
innerrightmargin=0pt,
innertopmargin=-6pt,
]{example}{Example}[section]

\usepackage{algorithm}
\usepackage[noend]{algpseudocode}

\usepackage{listings}
\usepackage{enumitem}
\setlist[enumerate]{itemsep=0mm}
\usepackage[nottoc,notlot,notlof]{tocbibind}

\newcommand{\N}{\mathbb{N}}

\newcommand{\R}{\mathbb{R}}

\newcommand{\W}{\mathbb{W}}
\newcommand{\X}{\mathbb{X}}
\newcommand{\Y}{\mathbb{Y}}
\newcommand{\Z}{\mathbb{Z}}

\newcommand{\CC}{\mathcal{C}}

\newcommand{\PP}{\mathcal{P}}

\renewcommand{\SS}{\mathcal{S}}

\newcommand{\XX}{\mathcal{X}}

\DeclarePairedDelimiterX{\normop}[1]{\lVert}{\rVert_{\mathrm{op}}}{#1}
\DeclarePairedDelimiterX{\normopp}[1]{\lVert}{\rVert_{\mathrm{op}}^p}{#1}
\DeclarePairedDelimiterX{\normhs}[1]{\lVert}{\rVert_{\mathrm{HS}}}{#1}
\DeclarePairedDelimiter\floor{\lfloor}{\rfloor}

\makeatletter
\def\upintkern@{\mkern-7mu\mathchoice{\mkern-3.5mu}{}{}{}}
\def\upintdots@{\mathchoice{\mkern-4mu\@cdots\mkern-4mu}%
	{{\cdotp}\mkern1.5mu{\cdotp}\mkern1.5mu{\cdotp}}%
	{{\cdotp}\mkern1mu{\cdotp}\mkern1mu{\cdotp}}%
	{{\cdotp}\mkern1mu{\cdotp}\mkern1mu{\cdotp}}}

\newcommand{\UpMultiIntegral}[1]{%
	\edef\ints@c{\noexpand\upintop
		\ifnum#1=\z@\noexpand\upintdots@\else\noexpand\upintkern@\fi
		\ifnum#1>\tw@\noexpand\upintop\noexpand\upintkern@\fi
		\ifnum#1>\thr@@\noexpand\upintop\noexpand\upintkern@\fi
		\noexpand\upintop
		\noexpand\ilimits@
	}%
	\futurelet\@let@token\ints@a
}
\makeatother

\DeclareFontFamily{OMX}{mdbch}{}
\DeclareFontShape{OMX}{mdbch}{m}{n}{ <->s * [0.8]  mdbchr7v }{}
\DeclareFontShape{OMX}{mdbch}{b}{n}{ <->s * [0.8]  mdbchb7v }{}
\DeclareFontShape{OMX}{mdbch}{bx}{n}{<->ssub * mdbch/b/n}{}

\DeclareSymbolFont{uplargesymbols}{OMX}{mdbch}{m}{n}
\SetSymbolFont{uplargesymbols}{bold}{OMX}{mdbch}{b}{n}
\DeclareMathSymbol{\upintop}{\mathop}{uplargesymbols}{82}
\DeclareMathSymbol{\upointop}{\mathop}{uplargesymbols}{"48}

\DeclareFontEncoding{MDB}{}{}
\DeclareFontFamily{MDB}{mdbch}{}
\DeclareFontShape{MDB}{mdbch}{m}{n}{ <->s * [0.8]  mdbchrmb }{}
\DeclareFontShape{MDB}{mdbch}{b}{n}{ <->s * [0.8]  mdbchbmb }{}
\DeclareFontShape{MDB}{mdbch}{bx}{n}{<->ssub * mdbch/b/n}{}
\DeclareFontSubstitution{MDB}{cmr}{m}{n}
\DeclareSymbolFont{mathdesignB}{MDB}{mdbch}{m}{n}%
\SetSymbolFont{mathdesignB}{bold}{MDB}{mdbch}{b}{n}%
\DeclareMathSymbol{\upintclockwise}{\mathop}{mathdesignB}{128}
\DeclareMathSymbol{\upointclockwise}{\mathop}{mathdesignB}{130}
\DeclareMathSymbol{\upointctrclockwise}{\mathop}{mathdesignB}{132}
\DeclareMathSymbol{\upoiint}{\mathop}{mathdesignB}{134}
\DeclareMathSymbol{\upoiiint}{\mathop}{mathdesignB}{136}

\makeatletter
\newcommand{\upint}{\DOTSI\upintop\ilimits@}
\newcommand{\upoint}{\DOTSI\upointop\ilimits@}
\makeatother

\renewcommand{\int}{\upint}

\usepackage{stackengine}

\newcommand{\dpvar}{d_{\text{p-var}}}
\newcommand{\WGp}{WG\Omega_p(V)}
\newcommand{\WG}[1]{WG\Omega_{#1}(V)}
\newcommand{\Gp}{G\Omega_p(V)}
\newcommand{\CGp}{C^{\text{p-var}}\big([0,T], G^{(n)}\big)}

\usepackage{authblk}

\title{\textsc{Topologies on unparameterised rough path space}}

\author{Thomas Cass\textsuperscript{1,2}, William F. Turner\textsuperscript{1}}

\affil{\small\textsuperscript{1} Department of Mathematics, Imperial College London\\ \textsuperscript{2} Institute for Advanced Study, United States of America}

\date{\today}

\usepackage[style=alphabetic, giveninits = true, maxbibnames = 99]{biblatex}
\usepackage{csquotes}
\begin{document}

\maketitle
\begin{abstract}
    The signature of a $p$-weakly geometric rough path summarises a path up to a generalised notion of reparameterisation. The quotient space of equivalence classes on which the signature is constant yields unparameterised path space. The study of topologies on unparameterised path space, initiated in \cite{CT_topologies} for paths of bounded variation, has practical bearing on the use of signature based methods in a variety applications. This note extends the majority of results from \cite{CT_topologies} to unparameterised weakly geometric rough path space. We study three classes of topologies: metrisable topologies for which the quotient map is continuous; the quotient topology derived from the underlying path space; and an explicit metric between the tree-reduced representatives of each equivalence class. We prove that topologies of the first type (under an additional assumption) are separable and Lusin, but not locally compact or completely metrisable. The quotient topology is Hausdorff but not metrisable, while the metric generating the third topology is not complete and its topology is not locally compact. We also show that the third topology is Polish when $p=1$.
\end{abstract}

\section{Introduction}
The signature of a $p$-weakly-geometric rough path is a group-like element of the extended tensor algebra that faithfully captures the effect of the path on non-linear control systems. The celebrated $2010$ article of Hambly and Lyons \cite{HL} showed that the signature of a continuous path of finite $1$-variation is unique up to tree-like equivalence; a generalised notion of reparameterisation. This result has since been extended to the general case of $p>1$ in \cite{BGLY}. Combined, these papers provide a well-defined notion of unparameterised path space, that is the quotient of $p$-weakly geometric rough paths by tree-like equivalence. The space of unparameterised paths then becomes a group with multiplication defined by path concatenation and inverses given by path reversal. 

As initially remarked in \cite{HL}, there are several natural choices of topology that one might impose on unparameterised path space, and this choice has a bearing on practical uses of signature methods. Indeed, the universal approximation theorem of \cite{LLN} loosely states that any continuous function on a compact subset of unparameterised path space may be uniformly approximated by linear functionals on the signature. Of course, the abundance of compact sets and continuous functions is entirely determined by the imposed topology. Moreover, much of the recent success that the signature has found in application is through signature-based kernel methods, see \cite{KO,SigPDE,WSK,SKlimit,CT_free} for various choices of signature-based kernels. Several tools used to study universality and characteristicness of kernels relies on the topological assumption that the underlying space is locally compact; a property that is especially useful in the study of translation invariant kernels on topological groups via Fourier transforms.  Finally, if one intends to do probability theory on these spaces then a Polish topology is desirable, as this is the setting of many standard tools such  Prokhorov's Theorem and Ulam's Thoerem.

A detailed study of the topological properties of three natural choices was conducted for $p=1$ in \cite{CT_topologies}. The aim of this note is to extend the results of \cite{CT_topologies} to $p>1$. We study three classes of topologies on unparameterised path space
\begin{enumerate}[label=\arabic*)]
    \item Topologies induced by metrics on (suitable subsets) of the extended tensor algebra via the injectivity of the signature map on unparameterised path space;
    \item The quotient topology inherited from the $p$-variation topology on the underlying path space;
    \item A metric topology defined as the $p$-variation distance between certain tree-reduced representatives of each equivalence class.
\end{enumerate}
Our main results are summarised as follows.
\begin{enumerate}[label=\arabic*)]
    \item The topologies are strictly order with the induced topologies being the weakest, and the metric topology the strongest;
    \item All three topologies are Hausdorff;
    \item Under additional assumptions, the induced topologies are separable, but not completely metrisable or locally compact, but are $\sigma$-compact and Lusin spaces;
    \item The quotient topology is not metrisable;
    \item The metric generating the metric topology is not complete.
\end{enumerate}
We also resolve and open question of \cite{CT_topologies} and prove that the metric topology is completely metrisable, and hence Polish, when $p=1$. The proof of this fact relies on a characterisation of tree-reduced paths as the unique minimisers of length in each equivalence class. This characterisation fails to hold when $p>1$, and so the complete metrisability in this setting remains an open question. Finally, we introduce the totality of unparameterised rough path space as the union of all unparameterised paths. We show that any (reasonable) metrisable topology on this space inherits the same properties as the induced topologies for a fixed $p$.

We comment here on two open question that we have been unable to resolve. Firstly, are the topologies separable of the topologies when $p>1$? This is a more subtle question than the case $p=1$. When $p=1$, it is always possible to reparameterise a continuous path of bounded variation to be Lipschitz, which forms a separable subset in the $1$-variation topology). By contrast, the analogous statement for $p>1$ fails to hold. Secondly, is the quotient topology completely regular? An affirmative answer to this question would allow for its use in study of expected signatures in the setting of \cite{JMLR-CO}.

\subsection{Outline}
In \cref{sec: rough_paths} we introduced the prerequisite information rough path theory. This includes the extended tensor algebra and group-like elements; the notion of $p$-variation; and (weakly)-geometric rough paths and their signature. We then recall several standard properties of (weakly)-geometric rough paths that are central to our analysis. Finally we define piecewise linear paths and a special subset thereof, namely axis paths. We introduce unparameterised path space in \cref{sec: unparameterised_paths} together with the topologies that we consider. The properties of these topologies is discussed in \cref{sec: results}, while the totality of unparameterised paths is considered in \cref{sec: totality}.




\section{Rough Path Spaces and Signatures}\label{sec: rough_paths}
\subsection{Weakly Geometric Rough Paths}
We follow \cite{BGLY} for the relevant definitions from rough path theory. Let $V$ be a finite-dimensional vector space with $\text{dim}(V)\geq 2$, which equipped with inner product $\langle\cdot,\cdot\rangle$ and corresponding norm $\norm{\cdot}$. Let the tensor powers of $V$, $\big(V^{\otimes n}, n\geq 1\big)$ be equipped with a family of admissible norms $\norm{\cdot}_{V^{\otimes n}}$ satisfying:
\begin{enumerate}[label = \arabic*)]
    \item For all $m,n\in\N$ and $u\in V^{\otimes m}$, $v\in V^{\otimes n}$ it holds that
    \[
    \norm{u\otimes v}_{V^{\otimes (m+n)}}\leq \norm{u}_{V^{\otimes m}}\norm{v}_{V^{\otimes n}}.
    \]
    \item For any permutation $\sigma \in \Sigma_n$ and $u_1,\dots,u_n\in V$
    \[
    \norm{u_1\otimes \dots \otimes u_n}_{V^{\otimes n}}=\norm{u_{\sigma(1)}\otimes\dots\otimes u_{\sigma(n)}}_{V^{\otimes n}}.
    \]
\end{enumerate}
Let $T((V))$ be the extended tensor algebra over $V$, that is
\[
T((V))\coloneqq \prod_{n=0}^\infty V^{\otimes n},
\]
and write $G^{(\star)}$ for the subset of group-like elements. Let $\mathbf{1}$ be the unit element of $T((V))$. The tensor product on these spaces is the standard extension of the tensor product on the projections
\begin{align*}
    \pi^{(n)}:&T((V))\to T^{(n)}(V)\coloneqq \prod_{i=0}^n V^{\otimes i};\\
    \pi_n:&T((V))\to V^{\otimes n}.
\end{align*}
Recall that the projection of $G^{(\star)}$, under $\pi^{(n)}$, recovers $G^{(n)}\coloneqq\pi^{(n)}\big(G^{(\star)}\big)$, the step-$n$ free nilpotent Lie group. Equip $G^{(n)}$ with Carnot-Caratheodory metric $d$, see \cite[Chapter 7]{FV} for its definition and properties.

\begin{definition}
    Let $G^{(\star)}_{\text{p.r.c}}$ denote the the $u\in G^{(\star)}$ for which
    \[
    \max_{i\in\N}\norm{\pi_i(u)}^{\frac{1}{i}}<\infty.
    \]

\end{definition}
From now on, we will omit ``$\otimes$'' and simply write $uv$ for $u\otimes v$.
\begin{definition}
    Let $(E,\rho)$ be a metric space, then we say that $\gamma:[0,T]\to E$ has finite $p$-variation if 
    \begin{equation}
        \norm{\gamma}_{\text{p-var};[0,T]}\coloneqq\sup_{\PP}\left(\sum_{t_i\in\PP}\rho\big(\gamma_{t_i},\gamma_{t_{i+1}})^p\right)^{\frac{1}{p}}<\infty.
    \end{equation}
\end{definition}
Weakly-geometric rough paths will take their values in the step-$n$ free nilpotent Lie group.
\begin{definition}[Weakly geometric rough paths]
    We let $C^{\text{p-var}}\big([0,T], G^{(n)}\big)$ be the space of continuous paths finite $p$-variation such that $\gamma_0$ is the unit element of $G^{(n)}$. The space of weakly geometric rough paths $WG\Omega_p(V)$ is defined to be $C^{\text{p-var}}([0,T], G^{(\floor{p})})$
\end{definition}
When $n=1$, we may trivially identify $G^{(1)}$ with $V$, so that $C^{\text{p-var}}\big([0,T], G^{(n)}\big)\cong C^{\text{p-var}}\big([0,T], V\big)$. We will make frequent use of this fact, especially in the case $p=1$, and will not distinguish between the two spaces. We note here an important consequence of Carnot-Caratheodory metric.
\begin{proposition}\label{prop: p_var_incr_n}
    Let $\X^{\leq n}\in\CGp$, then the $p$-variation norm of $\X^{\leq n}$ in $\CGp$ is greater than or equal to the $p$-variation norm of the projection $\X^{\leq m}\in C^{\text{p-var}}\big([0,T],G^{(m)}\big)$ for all $m\leq n$.
\end{proposition}
\begin{definition}[$p$-variation distance]
    For $\X^{\leq n},\Y^{\leq n}\in\CGp$ we define the $p$-variation distance by
    \begin{equation}\label{eq: pvar_metric}
    d_{\text{p-var}}^{(n)}\big(\X^{\leq n},\Y^{\leq n}\big) = \max_{1\leq i\leq n} \sup_\PP\left(\sum_{t_j\in\PP} \norm{\pi_i\Big(\big(\X_{t_j}^{\leq n}\big)^{-1}\big(\X_{t_{j+1}}^{\leq n}\big) - \big(\Y_{t_j}^{\leq n}\big)^{-1}\big(\Y_{t_{j+1}}^{\leq n}\big)\Big)}^{\frac{p}{i}}\right)^{\frac{i}{p}},
\end{equation}
and equip $\CGp$ with the topology induced by $d_{\text{p-var}}^{(n)}$. 
\end{definition}
For two weakly geometric rough paths $\X^{\leq \floor{p}},\Y^{\leq \floor{p}}\in\WGp$, we will omit the dependence on $n$ and simply write $\dpvar\big(\X^{\leq \floor{p}},\Y^{\leq \floor{p}}\big)\coloneqq d_{\text{p-var}}^{(\floor{p})}\big(\X^{\leq \floor{p}},\Y^{\leq \floor{p}}\big)$. Additionally, when the level to which a path is defined is clear from the context, we will write $\X\equiv \X^{\leq \floor{p}}$. The following theorem of Lyons' informally states that for continuous a path of finite $p$-variation, if the first $\floor{p}$ iterated integrals are provided as input (subject to certain algebraic and analytic constraints), then the remaining stack of iterated integrals is uniquely determined. This result may be found in \cite[Theorem 9.5]{FV} for finite-dimensional $V$ and \cite[Corollary 3.9]{CDLL_integration} for infinite dimensional $V$, see also \cite[Theorem 2.2.1]{Lyons1998}.
\begin{theorem}[Lyons' Extension Theorem]
    Let $\X^{\leq \floor{p}}\in WG\Omega_p(V)$, then there exists a unique path $\SS_{0,\cdot}\big(\X^{\leq \floor{p}}\big):[0,T]\to G^{(\star)}_{\text{p.r.c}}$ such that 
    \begin{enumerate}[label=\arabic*)]
        \item $\SS_{0,0}\big(\X^{\leq \floor{p}}\big)=\mathbf{1}$;
        \item $\pi^{(n)}\Big(\SS_{0,t}\big(\X^{\leq \floor{p}}\big)\Big)$ defines a continuous map from $\WGp$ to $\CGp$;
        \item $\pi^{(\floor{p})}\Big(\SS_{0,t}\big(\X^{\leq \floor{p}}\big)\Big)=\X^{\leq \floor{p}}_t$.
    \end{enumerate}
    We call $\SS\big(\X^{\leq \floor{p}}\big)\coloneqq\SS_{0,T}\big(\X^{\leq \floor{p}}\big)$ the signature of $\X^{\leq \floor{p}}$.
\end{theorem}
We note that the Lyons' lift defines a canonical mapping $\WGp\hookrightarrow\CGp$ for $n\geq\floor{p}$ given by
\[
\X^{\leq \floor{p}}\mapsto \X^{\leq n}\coloneqq \pi^{(n)}\Big(\SS_{0,\cdot}\big(\X^{\leq\floor{p}}\big)\Big).
\]
Additionally, we will write $\X^{<\infty}\coloneqq \SS_{0,T}\big(\X^{\leq \floor{p}}\big)$. We denote by $\mathscr{S}_p$ the image of $\WGp$ in $G^{(\star)}_{\text{p.r.c}}$ under the map $\SS$ and write
\[
\mathscr{S}\coloneqq\bigcup_{p\geq 1} \mathscr{S}_p\subseteq G^{(\star)}_{\text{p.r.c.}},
\]
for the space of all signatures. 
\begin{remark}
    For $p\in [1,2)$, the signature of $\gamma\in WG\Omega_p(V)$ may be realised as the sequence of $n$-fold iterated Young integrals given by
\begin{align*}
    \SS_{0,t}(\gamma)^n&\coloneqq \int\limits_{0\leq t_1\leq\dots\leq t_n\leq t}\dif\gamma_{t_1}\dots\dif\gamma_{t_n}\in V^{\otimes n};\\
    \SS_{0,t}(\gamma) &\coloneqq \big(1,\SS_{0,t}(\gamma)^{1},\SS_{0,t}(\gamma)^{2},\dots\big)\in T((V)).
\end{align*}
\end{remark}
An important subset of weakly geometric rough paths are the geometric ones. This is the subset of paths of $\WGp$ which may be approximated in $\dpvar$ by paths of finite $q$-variation for some $q<p$. Formally, we use the following definition.
\begin{definition}
    The space of geometric rough paths $G\Omega_p(V)\subset \WGp$ is the set of continuous paths $\X^{\leq \floor{p}}:[0,T]\to G^{(\floor{p})}$, starting at the unit, for which there exists a sequence of paths $\big(\X^{\leq 1}_n\big)_{n=1}^\infty\subset WG\Omega_1(V)$ with
    \[
    \dpvar\Big(\X^{\leq \floor{p}}_n,\X^{\leq \floor{p}}\Big)\to 0.
    \]
\end{definition}
\begin{remark}\label{rem: p_1_geometric}
    Clearly in the case $p=1$, we have by definition that $\WG{1}=G\Omega_1(V)$. However, the equivalent notion of ``geometric'' rough paths is the space $C^{0,1-\text{var}}\big([0,T], V\big)$ of absolutely continuous paths, given by the closure of smooth paths in the $1$-variation topology.
\end{remark}

\subsection{Basic Properties of Rough Paths and Signatures}
We record here several standard, but important, facts concerning the spaces $\Gp$, $\WGp$ and their signatures. Firstly, we note that $\Gp$ is a closed and separable subset of $\WGp$, see \cite[Proposition 8.25, Corollary 1.35]{FV} for a proof.
\begin{lemma}\label{lem: pw_density}
    The spaces $\Gp$ for $p>1$, and $C^{0,1-\text{var}}\big([0,T], V\big)$ for $p=1$, are Polish spaces with the lifts of piecewise linear paths forming a dense subset.
\end{lemma}
The following lemma will be crucial in generalising the results of \cite{CT_topologies} to the setting $p>1$.
\begin{lemma}[Properties of the $p$-variation metric]\label{lem: pq_var}
    Fix $q\geq p$.
    \begin{enumerate}[label=\arabic*)]
        \item Suppose $\X^{\leq n},\Y^{\leq n}\in \CGp$, then
            \[
            d_{\text{q-var}}^{(n)}\big(\X^{\leq n},\Y^{\leq n}\big)\leq d_{\text{p-var}}^{(n)}\big(\X^{\leq n},\Y^{\leq n}\big).
            \]
        \item If $\X^{\leq \floor{p}}_n\to\X^{\leq \floor{p}}$ in $\WGp$, then $\X^{\leq \floor{q}}_n\to\X^{\leq \floor{q}}$ in $WG\Omega_q(V)$.
        \item We have the strict inclusions $\Gp\subset\WGp\subset G\Omega_q(V)$.
    \end{enumerate}
\end{lemma}
\begin{proof}
    The first item is a consequence of the definition of \cref{eq: pvar_metric} and the inequality
    \[
    \Big(\sum\abs{a_j}^\frac{q}{i}\Big)^{\frac{i}{q}}\leq \Big(\sum\abs{a_j}^\frac{p}{i}\Big)^{\frac{i}{p}}.
    \]
    For the second item, Lyons' lift defines a continuous map from $\WGp\hookrightarrow \CGp$ for $n\geq \floor{p}$, see \cite[Corollary 9.11]{FV}. The claimed convergence is a consequence of this continuity and the first item. The final item is standard, see \cite[Corollary 8.24]{FV}.
\end{proof}
\begin{definition}
     For every $\X\in WG\Omega_p(V)$, we define the control $\omega_{\X, p}(s,t)\coloneqq \big\|\X\big\|_{\text{p-var};[s,t]}^p$. Let 
    \[
    \varphi(t) \coloneqq\frac{\omega_{\X, p}(0,t)}{\omega_{\X, p}(0,T)}T,
    \]
    and $\tilde{\X}$ the path implicitly defined by the relation $\X = \tilde{\X} \circ \varphi$. We call $\tilde{\X}$ the  ``H\"older-control reparameterisation'' of $\X$.
\end{definition}
\begin{proposition}\label{prop: constant_control}
    For any $\X\in WG\Omega_p(V)$, its H\"older-control reparameterisation $\tilde{\X}$ satisfies
    \begin{enumerate}[label = \arabic*)]
        \item $\big\|\tilde{\X}\big\|_{\text{p-var};[0,t]}^p=\frac{t}{T}\big\|\tilde{\X}\big\|_{\text{p-var};[0,T]}^p$;
        \item $\tilde{\X}$ is $\tfrac{1}{p}$-H\"older continuous with $\big\|\tilde{\X}\big\|_{\frac{1}{p}\text{-H\"ol}}\leq \frac{1}{T^{\nicefrac{1}{p}}}\big\|\tilde{\X}\big\|_{\text{p-var};[0,T]}$;\
    \end{enumerate}
\end{proposition}
\begin{proof}
    Both points are well-known and can be found in \cite[Proposition 5.14]{FV}, \cite{SaintFlour} or \cite[Lemma 3.2.2]{Chevyrev_thesis}.
\end{proof}
\begin{remark}
    Note that if $\X\in\WGp$ satisfies $\norm{\X}_{\text{p-var};[0,t]}^p=\frac{t}{T}\norm{\X}_{\text{p-var};[0,T]}$, then $\X=\tilde{\X}$, since $\phi(\cdot)$ will be the identity.
\end{remark}
\begin{remark}
    In the case $p=1$, H\"older-control parameterisation is equivalent to constant-speed parameterisation, i.e. $\vert\gamma_t^{\prime}\vert=\frac{1}{T}\norm{\gamma}_{1\text{-var};[0,T]}$ almost-everywhere. Indeed, if $\gamma$ has H\"older-control parameterisation then its weak derivative $\gamma^\prime\in L^1([0,T],V)$ is uniquely defined and must satisfy
    \[
    \frac{t}{T}\norm{\gamma}_{1\text{-var};[0,T]}=\norm{\gamma}_{1\text{-var};[0,t]}=\int_0^t\vert\gamma^\prime_u\vert\dif u.
    \]
    Hence $\vert\gamma^\prime_u\vert=\frac{1}{T}\norm{\gamma}_{1\text{-var};[0,T]}$ almost-everywhere by the Lebesgue Differentiation Theorem.
\end{remark}
An important concept in the study of signatures is that of path-concatenation and reversal.
\begin{definition}
Given $\X,\Y\in \WGp$, we define their concatenation as 
\[
(\X\star\Y)(t)\coloneqq
\begin{cases}
    \X\big(\tfrac{2t}{T}\big),\quad &t\in \big[0, \tfrac{T}{2}\big],\\
    \X\big(T\big)\Y\big(\tfrac{2t-T}{2}\big),\quad &t\in \big[\tfrac{T}{2}, T\big],
\end{cases}
\]
and the reversal of a path by
\[
\overleftarrow{\X}_t\coloneqq \big(\X_T\big)^{-1}\X_{T-t}.
\]
\end{definition}
The following is a well-known fact, whose proofs follows essentially the same proof as the well known fact that $C^{\text{p-var}}\big([0,T], V\big)$ is a Banach space under the $p$-variation norm.
\begin{lemma}
    Fix $p\geq 1$, and let $\X,\Y\in \WGp$ then
    \[
    \norm{\X\star\Y}_{\text{p-var}}\leq \norm{\X}_{\text{p-var}}+\norm{\Y}_{\text{p-var}}.
    \]
\end{lemma}
\begin{proof}
    Fix $\PP$ to be a partition of $[0,T]$ and let $t_j\in\PP$ be such that $\frac{1}{2}\in [t_j, t_{j+1}]$. Notice that
    \begin{equation}\label{eq: p_sub_additive}
    d\big((\X\star\Y)_{t_j},(\X\star\Y)_{t_{j+1}}\big)=d\big(\X_{2tj}, \X_T\Y_{2t_{j+1}-T}\big)\leq d\big(\X_{2t_j}, \X_T\big)+d\big(\Y_0,\Y_{2t_{j+1}-T}\big),
    \end{equation}
    where the inequality follows from the triangle inequality and the definition of concatenation. Then
    \begin{align*}
        \left(\sum_{t_i\in\PP}d\big((\X\star\Y)_{t_i},(\X\star\Y)_{t_{i+1}}\big)^p\right)^{\frac{1}{p}}\leq & \left(\sum_{t_i<t_j}d\big(\X_{2t_i},\X_{2t_{i+1}})^p +d\big(\X_{2t_j}, \X_T\big)^p\right)^{\frac{1}{p}}\\
        &+\left(d\big(\Y_0,\Y_{2t_{j+1}-T}\big)^p+\sum_{t_i>t_j}d\big(\Y_{2t_i-T},\Y_{2t_{i+1}-T}\big)^p\right)^{\frac{1}{p}}\\
        \leq& \norm{\X}_{\text{p-var}}+\norm{\Y}_{\text{p-var}},
    \end{align*}
    where the first inequality follows \cref{eq: p_sub_additive} and the sub-additivity of $\ell_p$ norms. The second inequality follows from taking the supremum over partitions of the right-hand-side. We conclude by taking supremums over partitions $\PP$ on the left-hand-side.
\end{proof}
The final result of this subsection is commonly known as Chen's identity \cite{Chen1} and states that the set of signatures is closed under multiplication and inverses. We will denote by $o$ the constant path in $\WGp$ for which $o(t)=\mathbf{1}$ for all $t\in [0,T]$.
\begin{proposition}
For $\X,\Y\in \WGp$, the signature map satisfies
\begin{enumerate}[label = \arabic*)]
    \item Chen's identity: $(\X\star\Y)^{<\infty}=\X^{<\infty}\Y^{<\infty}$;
    \item Inverses: $\X^{<\infty}\overleftarrow{\X}^{<\infty}=o^{<\infty}=\mathbf{1}$.
\end{enumerate}
\end{proposition}
\subsection{Piecewise Linear Paths}
Central to the analysis in this article is the lift of piecewise linear paths, and in particular axis paths. Given a vector $v\in V$, we let $\gamma_v\in WG\Omega_1(V)$ be the straight line defined by $\gamma_v(t)=\frac{t}{T}v$. For a collection of vectors $v_1,\dots,v_n\in V$, we define the piecewise linear path $\X^{\leq 1}\in WG\Omega_1(V)$ by $\X^{\leq 1}\coloneqq \gamma_{v_1}\star\gamma_{v_2}\star\dots\star\gamma_{v_n}$ with $1$-variation $\big\|\X^{\leq 1}\big\|_{1\text{-var}}=\abs{v_1}+\dots+\abs{v_n}$. We will often reparameterise $\X^{\leq 1}$ to have constant $1$-H\"older-control variation, i.e. parameterised at constant speed. An important example of piecewise linear paths are the so-called axis paths; these satisfy $\langle v_i,v_{i+1}\rangle = 0$ for $i=1,\dots, n-1$.
\section{Unparameterised Rough Path Spaces}\label{sec: unparameterised_paths}
The terminal time solution of a rough-differential equation (RDE)
\begin{equation}
    \dif Y_t=f(Y_t)\dif \X_t,\ Y_0 = y,
\end{equation}
is invariant to not only reparameterisations of of $\X$, but also pieces of $\X$ which retrace themselves. The notion of tree-like equivalence formalises these concepts.
\begin{definition}\label{def: tree_equivalence}
    A path $\X\in\WGp$ is called tree-like if there exists a real tree $\mathcal{T}$ such that $\X$ admits a factorisation $\X=\phi \circ \rho$ through a pair of continuous maps $\rho:[0,T]\mapsto \mathcal{T}$ and $\phi:\mathcal{T} \mapsto V$, where $\rho(0)=\rho(T)$. For two paths $\X,\Y\in\WGp$, if $\X\star\overleftarrow{\Y}$ is tree-like, we call $\X$ and $\Y$ tree-like equivalent and write $\X\sim_\tau\Y$.
\end{definition}
A path devoid of tree-like pieces is called tree-reduced.
\begin{definition}[\cite{BGLY}]
A path $\X\in WG\Omega_p(V)$ is called tree-reduced if the path $t\mapsto\SS_{0,t}(\X)\in G^{(\star)}_{\text{p.r.c}}$ is injective.
\end{definition}
Tree-reduced paths prove a fundamental object in the proof of the uniqueness of the signature in \cite{BGLY} and also form a core concept in our study of topologies on unparameterised path space.
The fact that the signature is unique up to tree-like equivalence was first shown in the case $p=1$ in \cite{HL} via the study of the Hyperbolic development map. More recently, this result has been extended to weakly-geometric rough paths in \cite{BGLY} through the study of one-forms on truncated signature paths.
\begin{theorem}[\cite{BGLY}]
    Let $p\geq 1$, then
    \begin{enumerate}[label=\arabic*)]
        \item The relation $\sim_\tau$ defines an equivalence relation on $\WGp$;
        \item For all $\X,\Y\in\WGp$ it holds that $\X^{<\infty}=\Y^{<\infty}$ if and only if $\X\sim_\tau\Y$.
        \item Each $\sim_\tau$ equivalence class contains a unique tree-reduced representative (up to reparameterisation).
    \end{enumerate}
\end{theorem}
The space of unparameterised paths provides an object on which the signature map is injective.
\begin{definition}[Unparameterised Path Space]
    The space of unparameterised rough paths $\CC_p$ is defined as the quotient space $WG\Omega_p/\sim_\tau$. The quotient map is denoted by $\pi$.
\end{definition}
We use $[\X]$ to denote the tree-like equivalence class of a path $\X$, so that $\pi(\X)=[\X]$. We note that for $p>1$, if $\X$ is the tree reduced representative of an equivalence class, then it satisfies \cite{BGLY}
\[
\norm{\X}_{\text{p-var}}=\min_{\Y\in [\X]}\norm{\Y}_{\text{p-var}},
\]
but is not necessarily the unique path that attains the minimum. Consider the following example found in \cite{Driver_Notes}.
\begin{example}
    Consider $V\equiv \R^2$, $p\in (1,2)$ and $\varepsilon>0$ suitably small. Let $(e_i)_{i=1}^2$ be the Cartesian basis of $V$ and define
    \[
    \gamma = \gamma_{(1+\varepsilon) e_1}\star\gamma_{-2\varepsilon e_1}\star \gamma_{(1+\varepsilon) e_1}.
    \]
    Then clearly $\gamma$ is not tree-reduced in the sense of \cref{def: tree_equivalence} and in the equivalence class $[\gamma_{2e_1}]$, but one may compute $\norm{\gamma}_{p\text{-var}}=\norm{\gamma_{2e_1}}_{p\text{-var}}=2$.
\end{example}
However, if $p=1$, the tree-reduced representative will be the unique minimiser of length in each equivalence class \cite{HL}. Examples of tree-reduced paths in $WG\Omega_1(V)$ include piecewise linear paths $\gamma_{v_1}\star\dots\star\gamma_{v_n}$ satisfying $\cos\big((<_{\theta}(v_i,v_{i+1})\big)\neq-1$ for all $i=1,\dots,n-1$ \cite{LX}. Such piecewise linear paths are called irreducible. In particular, axis paths are tree-reduced with an alternative proof provided in \cite{CT_topologies}. 
\subsection{Topologies on Unparameterised Geometric Rough Path Space}
\begin{definition}
    Fix $p\geq 1$. We consider the following topologies on unparameterised path space.
    \begin{enumerate}[label=\arabic*)]
        \item Let $\mathscr{m}$ be any metric on $\CC_p$ for which $\SS:WG\Omega_p\to (\CC_p, \mathscr{m})$ is continuous. Let $\chi_\mathscr{m}$ denote the topology on $\CC_p$ induced by $\mathscr{m}$.
        \item The quotient topology, $\chi_\tau$, on $\CC_p$ inherited from the underlying topology on $WG\Omega_p(V)$.
        \item The metric topology $\chi_\mathscr{d}$ defined by the metric
        \begin{equation}
            \mathscr{d}([\X], [\Y])\coloneqq d_{\text{p-var}}(\X^\star,\Y^\star),
        \end{equation}
        where $\X^\star$ is the H\"older-control parameterisation of the unique tree-reduced representative of $[\X]$.
    \end{enumerate}
\end{definition}
An example topology of the first type is the product topology on $T((V))$. By the continuity of the Lyons' lift to each level, the continuity of the evaluation map and the definition of convergence in a product topology, it follows that $\SS:\Gp\to T((V))$ is a continuous map.
\begin{remark}\label{rem: metric_subspace}
    An alternative way to define the metric topology is as the induced topology via the injective map $\CC_p\hookrightarrow\WGp$ for $p>1$, and $\CC_1\hookrightarrow C^{0,1-\text{var}}\big([0,T],V\big)$ for $p=1$, which takes an equivalence class $[\X]$ to the path $\X^\star$.
\end{remark}
\begin{remark}\label{rem: alternative_metric}
    A possible alternative metric between equivalence classes is defined by
    \begin{equation}
            d_\star([\X], [\Y])\coloneqq \big\|(\X \star \overleftarrow{\Y})^\star\big\|_{\text{p-var}},
    \end{equation}
    This metric is clearly symmetric and satisfies the triangle inequality since
    \begin{align*}
    \norm{(\X \star \overleftarrow{\Y})^\star}_{\text{p-var}}&=\norm{\Big((\X \star \overleftarrow{\Z})^\star\star (\Z \star \overleftarrow{\Y})^\star\Big)^\star}_{\text{p-var}}\\
    &\leq \norm{(\X \star \overleftarrow{\Z})^\star\star (\Z \star \overleftarrow{\Y})^\star}_{\text{p-var}}\\
    & \leq \big\|(\X \star \overleftarrow{\Z})^\star\big\|_{\text{p-var}} + \big\|(\Z \star \overleftarrow{\Y}^\star\big\|_{\text{p-var}}\\
    &=  d_\star([\X], [\Z]) +  d_\star([\Z], [\Y]),
    \end{align*}
     where the second inequality follows from the sub-additivity of $p$-variation. The resulting topology is limited in practical use however, since it fails to even be separable. Indeed, consider the set of canonical lifts of the $V$-valued paths $K=\big\{\tfrac{t}{T}v:\norm{v}=1\big\}$ into $G\Omega_p(V)$. For $\X^{\leq \floor{p}}\in K$, the definition of the $p$-variation norm and \cref{prop: p_var_incr_n} implies that
    \[
    \big\|\X^{\leq \floor{p}}\big\|_{p\text{-var}}^p\geq \big\|\pi_1\big(\X_{1}-\X_{0}\big)\big\|^p= \norm{v}^p=1.
    \]
    Given two distinct paths $\X,\Y\in K$ the path $\big(\X\star\overleftarrow{\Y}\big)^{\leq 1}$ is a piecewise linear irreducible path and hence tree-reduced, so that the path $\big(\X\star\overleftarrow{\Y}\big)^{\leq \floor{p}}$ is also tree-reduced. The distance between these paths is then lower bounded by
    \begin{align*}
        d_\star\big(\big[\X^{\leq \floor{p}}\big],\big[\Y^{\leq \floor{p}}\big]\big)^p&=\big\|\big(\X\star\overleftarrow{\Y})^{\leq p}\|_{\text{p-var}}^p\\
        &\geq\big\|\X^{\leq \floor{p}}\big\|_{p\text{-var}}^p+\big\|\Y^{\leq \floor{p}}\big\|_{p\text{-var}}^p\\
        &\geq2,
    \end{align*}
    by the super-additivity of controls. Since $K$ is an uncountable set, the topology induced by this metric cannot be separable.
\end{remark}
\section{Topological Properties}\label{sec: results}
The topologies we defined in the preceding are ordered as follows.
\begin{proposition}\label{prop: inclusions}
    We have the inclusions
    \begin{equation}\label{eq: inclusions}
        \chi_{\mathscr{m}}\subseteq \chi_\tau\subseteq \chi_{\mathscr{d}}.
    \end{equation}
\end{proposition}
\begin{proof}
    By continuity of $\SS$ on $(\CC_p,\chi_\tau)$ and the definition of $\chi_{\mathscr{m}}$, it follows immediately that $\chi_{\mathscr{m}}\subseteq \chi_\tau$. The second inclusion follows analogously to \cite[Proposition 3.3]{CT_topologies}.
\end{proof}
The above inclusions are in fact strict and may be shown directly as in \cite{CT_topologies} in the case $p=1$. For brevity, however, we will leave the strictness as a corollary of \cref{thm: metrisability}.
\begin{corollary}\label{cor: hausdorff}
    All three topologies are Hausdorff.
\end{corollary}
\begin{proof}
    Since $\chi_{\mathscr{m}}$ is a Hausdorff topology by construction, we immediately obtain that all three are Hausdorff.
\end{proof}
From this point forwards we will assume that $\chi_{\mathscr{m}}$ satisfies the following assumption.
\begin{assumption}\label{ass: pq_cty}
    There exists some $q\in (\floor{p},\floor{p}+1)$ for which $\SS:\Big(\WGp, d^{(\floor{p})}_{\text{q-var}}\Big)\to \big(\CC_p,\chi_{\mathscr{m}}\big)$ is continuous.
\end{assumption}
Our study of the finer topological properties of $\chi_{\mathscr{m}}$ relies heavily on the fact that under \cref{ass: pq_cty}, open sets must contain tree-reduced paths of arbitrarily large $p$-variation.
\begin{lemma}\label{lem: unbounded_balls}
    Suppose $\chi_{\mathscr{m}}$ satisfies \cref{ass: pq_cty}, then non-trivial open sets are unbounded in $p$-variation in the sense that
    \begin{equation}\label{eq: unbounded_pvar}
        \sup_{[\X]\in U}\big\|\X^\star\big\|_{\text{p-var}}=\infty,
    \end{equation}
    for any non-empty open $U$.
\end{lemma}
\begin{proof}
Let $U\in\chi_{\mathscr{m}}$ be an open neighbourhood of some point $[\X]\in\CC_p$. By \cref{ass: pq_cty,lem: pq_var,lem: pw_density}, $\pi^{-1}(U)$ contains the lift of some piecewise linear path path parameterised at constant speed. With abuse of notation, we also call this path $\X^{\leq 1}$. We now follow the argument of \cite{CT_topologies}. Let $\{v_1,v_2\}$ be two orthonormal vectors not collinear to the last linear piece of $\X^{\leq 1}$. For every $\varepsilon>0$, define the sequences of paths $\big(\Y_{n,\varepsilon}^{\leq 1}\big)_{n=1}^\infty$ and $\big(\Z_{n,\varepsilon}^{\leq 1}\big)_{n=1}^\infty$ by
\begin{align}
    \Y_{n,\varepsilon}^{\leq 1}&=\X^{\leq 1}\star\gamma_{(n+\varepsilon)v_1}\star\gamma_{-(n+\varepsilon)v_1},\\
    \Z_{n,\varepsilon}^{\leq 1}&=\X^{\leq 1}\star\gamma_{\varepsilon v_2}\star\gamma_{nv_1}\star\gamma_{-\varepsilon v_2}\star\gamma_{-nv_1},
\end{align}
both parameterised at constant $1$-variation. For every $n,\varepsilon$, it holds that $\Y_{n,\varepsilon}^{\leq \floor{p}}\in[\X]\subseteq\pi^{-1}(U)$, whereas $\Z_{n,\varepsilon}^{\leq 1}$ is tree-reduced by construction. Additionally,
\[
d_{1\text{-var}}\big(\X_{n,\varepsilon}^{\leq 1}, \Z_{n,\varepsilon}^{\leq 1})\leq 4\varepsilon,
\]
so that $\Z_{n,\varepsilon}^{\leq \floor{p}}\in\pi^{-1}(U)$ for every $n$ and some $\varepsilon$ by \cref{lem: pq_var} and continuity of the signature map $\SS:\WGp\to (\CC_p,\chi_{\mathscr{m}})$. Finally, notice that
\begin{align*}
    \big\|\Z_{n,\varepsilon}^{\leq \floor{p}}\big\|_{\text{p-var}}^p&\geq \big\|\Z_{n,\varepsilon}^{\leq 1}\big\|_{\text{p-var}}^p\\
    &\geq \big\|\X^{\leq 1}\|_{\text{p-var}}^p+2\varepsilon^p+2n^p.
\end{align*}
\end{proof}
\begin{proposition}\label{prop: compact_balls}
    Suppose that $\chi_{\mathscr{m}}$ satisfies \cref{ass: pq_cty}, then the balls $B_p(r)\coloneqq\{[\gamma]:\norm{\gamma^\star}_{\text{p-var}}\leq r\}$ are compact in $\chi_{\mathscr{m}}$.
\end{proposition}
The proof follows \cite[Proposition 4.2]{CT_topologies} and \cite[Proposition 5.7]{CL}.
\begin{proof}
    Since $\chi_{\mathscr{m}}$ is a metrisable topology, it is enough to show that $B_p(r)$ is sequentially compact. Let $\big([\X_n]\big)_{n=1}^\infty$ be a sequence in $B_p(r)$, then each tree-reduced representative $\X_n^\star$ is $\frac{1}{p}-$H\"older continuous with $\frac{1}{p}$-H\"older norm bounded by $r$. Consequently, the sequence $\big(\X_n^\star\big)_{n=1}^\infty\subset \WGp$ is equicontinuous and bounded, with uniformly bounded $p$-variation. And so, for any $p<q<\floor{p}+1$, a subsequence converges in $WG\Omega_q(V)$ to some path $\X\in \WGp$, see \cite[Proposition 8.17]{FV}. By \cite[Lemma 5.12]{FV}, $\norm{\X}_{\text{p-var}}\leq r$ so that $[\X]\in B_p(r)$. Finally, by \cref{ass: pq_cty}, we have the convergence $[\X_n]\to [\X]$ in $\chi_{\mathscr{m}}$.
\end{proof}
The preceding results show that we can construct large sets that are nowhere dense in the topology $\chi_m$. Taking a countable union of these closed sets allows us to conclude several topological properties.
\begin{corollary}\label{cor: baire_space}
    Let $\chi_{\mathscr{m}}$ satisfy \cref{ass: pq_cty}, then the space $(\CC_p,\chi_{\mathscr{m}})$ is 
    \begin{enumerate}[label=\arabic*)]
        \item Separable;
        \item $\sigma$-compact;
        \item A Lusin space;
        \item Not a Baire space;
        \item Not locally compact.
    \end{enumerate}
\end{corollary}
\begin{proof}
    The first item follows from \cref{ass: pq_cty,lem: pq_var} and the separability result \cref{lem: pw_density}. We may write
    \begin{equation}\label{eq: compact_union}
        \CC_p = \bigcup_{r=1}^\infty B_p(r).
    \end{equation}
    Thus the Hausdorff space $(\CC_p,\chi_{\mathscr{m}})$ is a countable union of compact sets [cf. \cref{prop: compact_balls}] with empty interior [cf. \cref{lem: unbounded_balls}] and so $\sigma$-compact and a Lusin space, but not a Baire space. The final item follows from the Baire category theorem.
\end{proof}
We now state one of our main results.
\begin{theorem}[Complete Metrisability]\label{thm: metrisability}
Let $\chi_{\mathscr{m}}$ satisfy \cref{ass: pq_cty}, then
    \begin{enumerate}[label = \arabic*)]
        \item $\chi_{\mathscr{m}}$ is not completely-metrisable;
        \item $\chi_\tau$ is not metrisable;
        \item The metric generating $\chi_{\mathscr{d}}$ is not complete. When $p=1$, however, $\chi_{\mathscr{d}}$ is completely-metrisable.
    \end{enumerate}
\end{theorem}
We will leave the proof of complete metrisability of the metric topology for $p=1$ to \cref{prop: chi_d_polish}. We record here that complete metrisability of $\chi_\mathscr{d}$ for $p>1$ remains unresolved.
\begin{proof}
    The first item is an immediate consequence of \cref{cor: baire_space} and the Baire category theorem. For the second item we follow the strategy of \cite[Theorem 3.8]{CT_topologies} and show that $\chi_\tau$ cannot both be first countable and regular and so not metrisable. Assume that $\chi_\tau$ has a countable neighbourhood basis $\{U_n\}_{n=1}^\infty$ of $[o]\in\CC_p$ consisting of, with loss of generality, open sets. Fix two orthonormal vectors $\{v_1,v_2\}$ and consider the sequence of tree-like paths in $WG\Omega_1(V)$ given by
    \[
    \Y_n^{\leq 1}\coloneqq \gamma_{n v_1}\star\gamma_{-n v_1}=\gamma_{n v_1}\star\overleftarrow{\gamma_{nv_1}}.
    \]
    Then each $\Y_n^{\leq \floor{p}}\in [o]\in\CC_p$, since the lift of a tree-like path is tree-like. As such
    \[
    r_n\coloneqq \sup \big\{r>0: B_{\dpvar}\big(\Y^{\leq \floor{p}}; r\big)\subseteq U_n\big\}>0.
    \]
    Define the strictly decreasing sequence of positive real numbers $\{\delta_n\}_{n=1}^\infty$ by
    \[
    \delta_1=\min\{r_1,1\}\quad \text{and}\quad \delta_{n+1}=\frac{1}{2}\min\{\delta_n,r_{n+1}\},
    \]
    for $n\geq 1$ so that $\pi\Big(B_{\dpvar}\big(\Y_n^{\leq \floor{p}};\delta_n\big)\Big)\subseteq U_n$ for each $n$. For $\varepsilon>0$ let $\Z^{\leq 1}_{n,\varepsilon}$ be the path
    \[
    \Z^{\leq 1}_{n,\varepsilon}=\gamma_{\varepsilon v_2}\star\gamma_{(n-\varepsilon)v_1}\star\gamma_{-\varepsilon v_2}\star\gamma_{-(n-\varepsilon)v_1},
    \]
    so that $d_{1\text{-var}}\big(\Z^{\leq 1}_{n,\varepsilon}, \Y_n^{\leq 1}\big)\leq 4\varepsilon$. And so by \cref{lem: pq_var}
    \[
    \lim_{\varepsilon\to 0}\dpvar\big(\Z_{n,\varepsilon}^{\leq\floor{p}}, \Y_n^{\leq\floor p}\big)=0.
    \]
    It then follows that $\Z_{n,\varepsilon}^{\leq\floor{p}}\in B_{\dpvar}\big(\Y_n^{\leq \floor{p}};\delta_n\big)$ for suitably small $\varepsilon_n$, resulting in $\Big[\Z_{n,\varepsilon_n}^{\leq\floor{p}}\Big]\in U_n$. Clearly
    \[
    \Big\|\Z_{n,\varepsilon_n}^{\leq\floor{p}}\Big\|_{\text{p-var}}\uparrow \infty,
    \]
    meaning that \[
    A\coloneqq \bigcup_{n=1}^\infty A_n,\text{ with }A_n\coloneqq\bigcup_{j=1}^{n}\Big\{\Big[\Z_{j,\varepsilon_j}^{\leq\floor{p}}\Big]\Big\},
    \]
    is closed in $\chi_\tau$. To see this, note that if $(\W_n)_{n=1}^\infty\subset \pi^{-1}(A)$ converges to some $\W\in\WGp$, then $\sup_n\norm{\W_n}_{\text{p-var}}<\infty$. Using the fact that each $\Z_{n,\varepsilon_n}^{\leq\floor{p}}$ is tree-reduced with exploding $p$-variation, there exists some $N$ for which $(\W_n)_{n=1}^\infty$ is contained in $A_N$, a finite union of equivalence classes each of which is closed by \cref{cor: hausdorff} and hence $\W\in\pi^{-1}(A)$. We have then constructed a closed set $A$, not containing $[o]$, but which has non-empty intersection with every neighbourhood of $[o]$; $\chi_\tau$ cannot then be both first countable and regular and so is not metrisable.

    For the final item let $\{v_1,v_2\}$ be a pair of unit vectors in $V$ which are not collinear. For each $\tfrac{1}{2}>\varepsilon > 0$ define the following tree-reduced path
    \begin{equation}\label{eq: piecewise_linear}
        \X_\varepsilon^{\leq 1}=\gamma_{\varepsilon v_2}\star\gamma_{(\nicefrac{1}{2}-\varepsilon)v_1}\star\gamma_{-\varepsilon v_1}\star\gamma_{-(\nicefrac{1}{2}-\varepsilon )v_1}.
    \end{equation}
    Let $\X^{\leq \floor{p}}$ be the canonical lift into $C^{1\text{-var}}\big([0,T], G^{(\floor{p})}\big)\subset \WGp$ which is then parameterised at $\tfrac{1}{p}$-H\"older control. Let $(\varepsilon_n)_{n=1}^\infty$ be a sequence satisfying $0<\varepsilon_n<\tfrac{1}{2}$ and converges to $0$. The sequence $(\X^{\leq 1}_{\varepsilon_n})_{n=1}^{\infty}$ is clearly bounded in $1$-variation and equicontinuous by the H\"older control reparameterisation. Hence the sequence $(\X^{\leq 1}_{\varepsilon_n})_{n=1}^{\infty}$ converges uniformly (along a subsequence) to some path $\X^{\leq 1}$ of finite $1$-variation, and so converges in $q$-variation for $q\in (1,2)$ by a standard interpolation argument. By the form of \cref{eq: piecewise_linear} it is clear that $\X^{\leq 1}$ must be a reparameterisation of the path $\gamma_{v_1}\star\gamma_{-v_1}$. By \cref{lem: pq_var} we conclude that $\X_{\varepsilon_n}^{\leq \floor{p}}\to \X^{\leq \floor{p}}$ in $\WGp$, and so forms a Cauchy sequence with respect to the metric $\mathscr{d}$. By continuity of the quotient map, We have the convergence of equivalence classes $\big[\X_{\varepsilon_n}^{\leq \floor{p}}\big]\to \big[\X^{\leq \floor{p}}\big]=[o]$ in $\chi_\tau$. By the uniqueness of limits in Hausdorff spaces, the only possible limit of $\big[\X_{\varepsilon_n}^{\leq\floor{p}}\big]$ in $\chi_d$ is then $[o]$. However, it is clear that
    \[
    \lim_{n\to\infty}\mathscr{d}\big(\big[\X^{\leq \floor{p}}_{\varepsilon_n}\big], [o]\big)\geq 2^{\nicefrac{1}{p}-1},
    \]
    so that $\big[\X^{\leq \floor{p}}_{\varepsilon_n}\big]$ does not converge to $[o]$ with respect to the $\mathscr{d}$.

    \end{proof}
As a result of the non-metrisability of the quotient topology, the inclusions from \cref{prop: inclusions} are strict.
\begin{corollary}
    The inclusions in \cref{eq: inclusions} are strict. That is
    \[
    \chi_{\mathscr{m}}\subset\chi_\tau\subset\chi_{\mathscr{d}}.
    \]
\end{corollary}
The non-completeness of the metric topology has the following consequence.
\begin{corollary}
    The space $(\CC_p,\chi_{\mathscr{d}})$ is not locally compact.
\end{corollary}
\begin{proof}
    Using the same techniques as in \cref{thm: metrisability}, we can show that $B_p(r)$ is not complete for any $r$. As a result, $[o]$ cannot have a compact neighbourhood in this topology as this would imply $B_p(r)$ is compact for some suitably small $r$.
\end{proof}
We conclude this section by positively answering an open question from \cite{CT_topologies}: is the metric topology $\chi_{\mathscr{d}}$ Polish when $p=1$?
\begin{proposition}\label{prop: chi_d_polish}
    When $p=1$, the space $(\CC_1,\chi_{\mathscr{d}})$ is Polish.
\end{proposition}
\begin{proof}
    Since $\chi_{\mathscr{d}}$ is a separable topology when $p=1$ \cite{CT_topologies}, it remains to show that it is completely metrisable. Recall the following facts:
    \begin{enumerate}[label=\arabic*)]
        \item A subspace of a completely metrisable space is completely metrisable if and only if it is $G_\delta$;
        \item For a topological space $\XX$, the set of continuity points of a function $f:\XX\to\R$ is $G_\delta$;
        \item A closed subset of a metric space is $G_\delta$;
        \item The intersection of two $G_\delta$ sets is $G_\delta$.
    \end{enumerate}
     Recall the space $C^{0,1\text{-var}}\big([0,T],V\big)$ of absolutely continuous paths endowed with the $1$-variation topology; this space is completely metrisable with the set of piecewise linear paths forming a dense subset. We split the proof into two components:
    \begin{enumerate}[label=\arabic*)]
        \item The set of continuity points of the map $\gamma\mapsto\|\gamma^{\star}\|_{1\text{-var}}$ is the set of tree-reduced paths;
        \item The set of paths parameterised at constant speed is closed in $C^{0,1\text{-var}}\big([0,T],V\big)$.
    \end{enumerate}
    For the first item, suppose that $\gamma$ is tree-reduced with $\gamma_n\to \gamma$ in $1$-variation, implying $\norm{\gamma_n}_{1\text{-var}}\to\|\gamma\|_{1\text{-var}}$. Suppose for a contradiction that $\norm{\gamma_n^\star}_{1\text{-var}}$ does not converge to $\|\gamma\|_{1\text{-var}}$; then there exists a subsequence (which we also denote by $\gamma_n^\star$) for which $\norm{\gamma_n^\star}\to\kappa\geq 0 $ where $\kappa < \norm{\gamma^\star}_{1\text{-var}}$. Since the sequence $(\gamma_n^\star)_{n=1}^\infty$ is equicontinuous and bounded, it holds by Arzel\`a-Ascoli that there exists a uniformly convergent subsequence converging to a path $\tilde{\gamma}$ with finite $1$-variation, for which $\gamma_n^\star\to \tilde{\gamma}$ in $\WGp$ for some $p\in (1,2)$. From \cite[Lemma 1.15]{FV}, we have the bound
    \[
    \norm{\tilde{\gamma}}_{1\text{-var}}\leq \liminf_{n\to\infty} \norm{\gamma_n^\star}_{1\text{-var}} = \kappa.
    \]
    By continuity of the signature map $\SS:\WGp\to T((V))$ (with $T((V))$ equipped with the product topology), we have the following
    \[
    \SS(\gamma)=\lim_{n\to\infty}\SS(\gamma_n)=\lim_{n\to\infty}\SS(\gamma_n^\star)=\SS(\tilde{\gamma}).
    \]
    And so $\tilde{\gamma}\in[\gamma]$, implying that $\gamma$ does not have shortest length in its equivalence class: a contradiction. For the reverse, suppose that $\gamma\in C^{0,1\text{-var}}\big([0,T],G^{(1)}\big)$ is not tree-reduced. By density, we can approximate $\gamma$ in $1$-variation by a sequence $(\gamma_n)_{n=1}^\infty$ of piecewise linear paths. By replacing any re-traced pieces with suitable approximations via axis paths, this sequence can be assumed to be tree-reduced. It follows that
    \[
    \lim_{n\to\infty}\norm{\gamma_n^\star}_{1{\text{-var}}} = \lim_{n\to\infty}\norm{\gamma_n}_{1{\text{-var}}} = \norm{\gamma}_{1{\text{-var}}}>\norm{\gamma^\star}_{1{\text{-var}}},
    \]
    so that the map is discontinuous at any path which is not tree-reduced.

    For the second part, let $(\gamma_n)_{n=1}^\infty$ be collection of paths parameterised at constant speed, which we recall is defined by $\norm{\gamma}_{1{\text{-var}}; [0,t]}=\frac{t}{T}\norm{\gamma}_{1{\text{-var}}; [0,T]}$. Suppose that $\gamma_n\to \gamma$ in $C^{0,1\text{-var}}\big([0,T],V\big)$, so that $\norm{\gamma_n}_{1{\text{-var}}}\to \norm{\gamma}_{1{\text{-var}}}$. It must also hold that $\gamma_n\vert_{[0,t]}$ converges to $\gamma\vert_{[0,t]}$ in $C^{0,1\text{-var}}\big([0,t],V\big)$. We conclude by by taking the limits
    \[
    \norm{\gamma}_{1{\text{-var}}; [0,t]}=\lim_{n\to\infty}\norm{\gamma_n}_{1{\text{-var}}; [0,t]}=\lim_{n\to\infty}\tfrac{t}{T}\norm{\gamma_n}_{1{\text{-var}}; [0,T]}=\tfrac{t}{T}\norm{\gamma}_{1{\text{-var}}; [0,T]}.
    \]
\end{proof}
It is unclear how one might generalise the proof of \cref{prop: chi_d_polish} to $p>1$; the proof relies on a characterisation of the tree-reduced element of an equivalence class in terms of its length. As discussed earlier, however, the equivalent characterisation fails to be true when $p>1$, and it remains an open question as to whether some other characterisation exists.
\section{The Totality of Unparameterised Rough Paths}\label{sec: totality}
We conclude this note with a discussion on the totality of all unparameterised rough paths $\CC$, that is the quotient of the space
\[
WG\Omega(V)\coloneqq \bigcup_{p\geq 1}\WGp
\]
by tree-like equivalence. Here for $q>p$, two paths $\X^{\leq \floor{p}}\in\WGp$ and $ \Y^{\leq \floor{q}}\in WG\Omega_q(V)$ are called tree-like equivalent if $\X^{\leq \floor{q}}$ and $\Y^{\leq \floor{q}}$ are tree-like equivalent. It is unclear how to topologise $WG\Omega(V)$ in a consistent fashion, and so defining a suitable quotient topology on $\CC$ is not straightforward. However, it is still possible to define topologies on $\CC$ induced by metrics on subspaces of $G^{(\star)}$.
\begin{remark}
    We can canonically identify $\CC$ with the space $\mathscr{S}$ and $\CC_p$ with $\mathscr{S}_p$, so that we can view $\CC_p$ as a subset of $\CC$ through the inclusion $\mathscr{S}_p\subset \mathscr{S}$. This will prove useful in the proceeding.
\end{remark}
In a similar spirit to \cref{sec: results} we have the following theorem.
\begin{theorem}
    Let $\chi$ be any metrisable topology on $\CC$ for which $\SS:WG\Omega_p(V)\to (\CC,\chi)$ is continuous for all $p$, then the space $(\CC, \chi)$ is
    \begin{enumerate}[label=\arabic*)]
        \item Separable;
        \item $\sigma$-compact;
        \item Not locally compact;
        \item Not polish.
    \end{enumerate}
\end{theorem}
\begin{proof}
   Recall that the (lift of the) same countable subset of piecewise linear paths is dense in each $\Gp$ for every $p\geq 1$, and that $\WGp\subset G\Omega_q(V)$ for all $q>p$. By continuity for each $p\geq 1$, the image of these paths under $\SS$ will then be dense in $\CC$.
   
   For the remaining points, we first claim that every open set in $\CC$ is unbounded in p-variation for some $p\geq 1$. First note that for every $p\geq 1$, the subspace topology on $\CC_p$ induced by the topology $\chi$ on $\mathscr{S}$ results in a topology on $\CC_p$ that satisfies \cref{ass: pq_cty}. Label this topology as $\chi_{\mathscr{m}}^p$. Let $U$ be a non-trivial open set in $\CC$ containing the signature of a weakly-geometric rough path $\X\in\CC_p$. Then $U\cap \CC_p\subseteq\CC_p$ will be a non-trivial open set in $(\CC_p,\chi_{\mathscr{m}}^p)$ and so unbounded in $p$-variation in the sense of \cref{eq: unbounded_pvar} by \cref{lem: unbounded_balls}. Let $B_p(r)$ be the set of equivalence classes whose tree-reduced representative may be viewed as an element of $\WGp$ with $p$-variation bounded by $r$. By the same argument as in \cref{prop: compact_balls}, each $B_p(r)$ will be compact in $(\CC,\chi)$. We may then write
   \[
   \CC = \bigcup_{p=1}^\infty\bigcup_{r=1}^\infty B_p(r).
   \]
   This is now a countable union of compact subsets of a Hausdorff space with empty interior, from which we deduce the remaining points by the Baire category theorem. 
\end{proof}
The set of group-like elements $G^{(\star)}$ is a closed and completely metrisable subset of $T((V))$ equipped with the product topology \cite[Lemma 2.6]{CT_topologies}, which itself is completely metrisable. Since $\mathscr{S}$ is contained in $ G^{(\star)}$ and is itself not completely metrisable, we obtain the following well known result.
\begin{corollary}
    The set of all signatures $\mathscr{S}$ is a strict subset of $G^{(\star)}$. If $G^{(\star)}$ is equipped with the product topology inherited from $T((V))$, then $\mathscr{S}$ is a meagre subset of $G^{(\star)}$.
\end{corollary}

\section*{Funding and Acknowledgements}
TC has been supported by the EPSRC Programme Grant EP/S026347/1 and acknowledges the support of the Erik Ellentuck Fellowship at the Institute for Advanced Study. WFT has been supported by the EPSRC Centre for Doctoral Training in Mathematics of Random Systems: Analysis, Modelling and Simulation (EP/S023925/1). For the purpose of open access, the authors have applied a Creative Commons Attribution (CC BY) licence to any Author Accepted Manuscript version arising.

The authors thank Xi Geng for suggesting the topology discussed in \cref{rem: alternative_metric} and for suggesting the study of the totality of unparameterised rough paths.
\printbibliography[heading=bibintoc,title={References}]
\clearpage
\appendix

\end{document}